\documentclass[12pt]{amsart}
\usepackage{amsmath} 
\usepackage{amssymb}
\usepackage{amsfonts}
\usepackage{a4}
\usepackage[T1]{fontenc}
\usepackage{ndstye}
\usepackage{comment}
\usepackage{mathtools}

\newcommand{\bc}{\mathbf C}

\newcommand{\br}{\mathbf R}
\newcommand{\bz}{\mathbf Z}
\newcommand{\bn}{\mathbf N}
\newcommand{\Cal}{\mathcal}

\newcommand{\gs}{\gtrsim}

\newcommand{\id}{\operatorname{Id}}

\newcommand{\im}{\operatorname{Im}}

\newcommand{\ls}{\lesssim}

\newcommand{\mn}[1]{\Vert#1\Vert}

\newcommand{\ol}{\overline}

\newcommand{\re}{\operatorname{Re}}
\newcommand{\restr}[1]{\big|_{#1}}

\newcommand{\set}[1]{\left\{\,#1\,\right\}}

\newcommand{\supp}{\operatorname{\rm supp}}

\newcommand{\w}[1]{\langle #1\rangle }

\newcommand{\wf}{\operatorname{WF}}

\newcommand{\wt}{\widetilde}

\newcommand{\wh}{\widehat}

\numberwithin{equation}{section}

\begin{document}
	
\baselineskip 19.95pt 
\lineskip 2pt
\lineskiplimit 2pt

\title[Local Solvability]{Local Solvability of quasilinear Pseudo-\\differential Operators of Real Principal Type}
\author[Nils Dencker]{Nils Dencker\\ \\Lund University}
\address{Centre for Mathematical Sciences, University of Lund, Box 118, S-221 00 Lund, Sweden}
\email{nils.dencker@gmail.com} 

\begin{abstract}
	In this paper we prove local solvability of quasilinear pseudodifferential operators which has homogeneous principal symbol of real principal type. This is a generalization of Theorem~A.1 in arXiv:2403.19054, which treated the case of quasilinear partial differential operators of order 2. The proof is by microlocalization to first order model operators.
\end{abstract}

\subjclass[2000]{35A01 (primary) 35S05, 58J40, 47G30 (secondary)}

\maketitle

\thispagestyle{empty}

\section{Introduction}

The solvability of differential operators has been an important issue since Hans Lewy \cite{lewy} presented a complex vector field that was not solvable anywhere. H\"ormander \cite{ho:nosolv} then showed that almost all linear differential operators are {\em not} solvable. After some development, Nirenberg and Treves \cite{nt:69} formulated their famous conjecture 1969 that the condition for solvability of a linear (pseudo)differential operator of principal type is condition~($ \Psi $). This is an invariant condition on the sign changes of the imaginary part of the principal symbol on the bicharacteristics of the real part. Principal type means that the principal symbol  vanishes of  first order on its zeros, so the bicharacteristics are not degenerate. The necessity of condition  ($ \Psi $) was proved by H\"ormander \cite{ho:nec} in 1980, the sufficiency was proved by the author \cite{de:nt} in 2006, and improved by Lerner \cite{ln:cutloss}. 

One important class of differential operators are those of {\em real principal type} which are the operators of principal type with real principal  symbols. For the linear PDO of real principal type, H\"ormander \cite{ho:thesis} proved solvability already in 1955 which he extended to pseudodifferential operators~\cite{HO:FIO2} in 1972. We shall extend this solvability result to quasilinear  pseudodifferential operators, showing local solvability with solutions having prescribed derivatives at a given point. The solution would in general not be  unique but the result can be used, e.g., for finding local solutions vanishing of some order at a point. For a special case, see the Appendix of~\cite{de:nsuff}.

To state the problem, let  $ f (x)\in C^\infty $ and consider the equation
\begin{equation}\label{pequation}
P(u,\partial^\alpha u, x, D) u = f, \quad \partial^\alpha u(x_0) = u_\alpha \quad |\alpha | < m
\end{equation}
where $ P $ is a nonlinear pseudodifferential operator in the $ x $ variables of order~$m \in \bn $ depending $ C^\infty $ on $ u $ and $\partial^\alpha u  $ for $  |\alpha | < m $. 
Then we have
\begin{equation}\label{psymbol}
P(v,x, D)  = p_m(v,x, D) + p_{m-1}(v,x, D) 
\end{equation}
modulo $ \Psi^{m-2} $ where $ p_j (v,x, \xi)  \in S^j  $ with  $ v = (u, \dots , \partial^\alpha u,\dots) 
\in C^\infty$ for  $ { |\alpha | < m} $.

We shall simplify the notation and write $ p(v, x, \xi)  $ instead of  $ p(u,\partial^\alpha u, x, \xi)  $ and $ \ol u $ instead of $(u_0, \dots, u_\alpha, \dots)  $, $ | \alpha  | < m $.
We shall assume that the principal symbol $ p_m $ is homogeneous of real principal type so that,  if $ \xi_0 \ne 0 $ then, after a linear symplectic change of variables, we have $\partial_\xi p_m \ne 0 $ when $ p_m = 0 $  at  $ ( \ol u,x_0,  \xi_0) $, and thus in a conical neighborhood, see \cite[Th. 4.3]{ho:weyl}. 
Here the symbols $ p_j(v,x,\xi) $ could be defined for real or complex valued $ v $ but need not be analytic.

We will use the classical Kohn-Nirenberg quantization having classical symbol expansions and  assume that $ P$ is properly supported so that $ P $ maps $  C^\infty $ to $  C^\infty $ and $  C^\infty_0 $ to $  C^\infty_0 $ after adding a term in $ \Psi^{-\infty}$.
Since the principal symbol is real valued, we shall also treat the case of real solvability.

\begin{defn}
	If  the symbols satisfy $p_j(v,x,  \xi) =  \ol  p_j(v,x,  -\xi) $, $ \forall \, j $, then $ P $ maps real valued functions to real valued functions, since $\ol{p(x,D)u} =  \ol{p}(x,-D)\ol u $, so $\re p(x,D)u = p(x,D)\re u $.
	Operators with this property will be called {\em real operators}. 
\end{defn}

The following is the main result of the paper.

\begin{thm}\label{appthm}
	Let $ X $ be a manifold and $ P \in \Psi^m(X)$ be a properly supported nonlinear operator given by~\eqref{psymbol}  so that $ p_j (v,x, \xi)  \in S^j(T^*X)  $ depends  $ C^\infty $ on $ \partial^\alpha v(x) \in   C^\infty  $, $ | \alpha  | < m $, and the principal symbol $ p_m $ is  homogeneous of degree $ m \in \bz_+ $.
	Let $ \ol u =  \set{ u_\alpha}_{ |\alpha | < m} $ and assume that $ p_m(v,x,\xi) $ is of real principal type in a neighborhood of $ ( \ol u, x_0) $. Then for any $ f  \in C^\infty(X) $  there exists a neighborhood $ U $ of   $ x_0 $ so that~\eqref{pequation} 
	has solution $u \in  C^\infty $ in $ U $. The  neighborhood~$ U $ will only depend on the bounds on $\ol u $,  $ f $ and the symbol of~$ P $.
	If $ P $ is a real operator of real principal type and the data $ f $ and $\ol u =   \set{ u_\alpha}_{ |\alpha | < m}  $  are real valued, then there is a real valued solution $u \in  C^\infty $  to~\eqref{pequation} in $ U $. 
	\end{thm}

Observe that the solution is not unique, for uniqueness one needs stronger conditions, for example,  hyperbolicity of $ P $ and initial values at a noncharacteristic hypersurface. The result will also hold if the symbol and the data depends $ C^\infty $ on a parameter $ w  \in \bc^k$  as in~\cite[Theorem~A.1]{de:nsuff}, which gives a special case of Theorem~\ref{appthm} for second order real nonlinear PDO of real principal type and real data. 
Theorem~\ref{appthm} will be proved in Section~\ref{thmpf} after a microlocal reduction to a nonlinear first order normal form. 

\begin{rem}
	It follows from the proof that Theorem~\ref{appthm} also holds if $ P(v,x,D) $ depends $ C^\infty $ on $\set{ u_\alpha}_{ |\alpha | < k}  $ for $ k  > m $ and $ p_m $ is of real principal type in a neighborhood of \/$ (x_0, \ol u ) $ where $ \ol u= \set{u_\alpha}_{| \alpha| < m}$, thus for any value of $u_\alpha  $ when $m  \le  | \alpha  | < k $. 
\end{rem}

\section{The Normal Form}

In this section we shall obtain a microlocal normal form of the operator. Since the solutions are local, we can reduce to the case  when $X = \br^n $.
For the proof of  Theorem~\ref{appthm} we shall solve the linear equation
\begin{equation}\label{plineq}
P(u,x, D) v(x) = f(x) \qquad \partial^\alpha v(x_0) = v_\alpha \quad |\alpha | < m
\end{equation}
which is a linear $ \Psi $DE  with $ u  \in C^\infty$ as parameter. 
By using iteration, estimates and compactness we shall obtain a solution to~\eqref{pequation}.

To solve the linear equation~\eqref{plineq} we shall microlocalize  in cones in the $ \xi $ variables. We say that a pseudodifferential operator (or Fourier integral operator)  $ a(v,x,D) $ depends  $ C^\infty $ on a parameter $ v(x) \in   C^\infty  $ if any $ C^\infty $ seminorm of the symbol (and phase function) is bounded by a finite number of  seminorms  of $ v $. 
For operators with symbols in $ S^{-\infty} $ this
means that the $ C^\infty $ kernel is a  $ C^\infty $ function of  $ v $.
Observe that compositions and adjoints of such operators also depend $ C^\infty $ on $ v $, see for example Lemma~\ref{contlem} and Remark~\ref{fiorem}.

\begin{defn}\label{gammadef}
	For any $ \varepsilon > 0 $  and $ \xi_0 \in \br^n $ such $ | \xi_0 | = 1 $ we let
	\begin{equation}\label{gammaform}
	\Gamma_{\xi_0, \varepsilon} = \set{\xi : \left| \xi /|\xi| - \xi_0  \right| < \varepsilon  }
	\end{equation}
	which is a conical neighborhood of  $ \xi_0 $.
\end{defn}

Recall that a partition of unity is a set $ \set{\phi_j}_j $ such that $ 0 \le \phi_j  \in C^\infty$ and $ \sum_j \phi_j \equiv 1 $.

\begin{rem}\label{microrem}
	For any $ \varepsilon > 0 $ small enough we can find a partition of unity on $ S^*\br^n $ and extend it by homogeneity in $ \xi $ to get a partition of unity $ \set{\varphi_j(\xi)}_j $ on $ T^*\br^n \setminus 0 $ such that  $0 \le \varphi_j \in S^0 $ is homogeneous and supported in $ \Gamma_{\xi_j, \varepsilon} $ for some $ |\xi_j| = 1 $. We can also find  $ \set{\psi_j}_j $ such that $ 0 \le \psi_j \in S^0$ is homogeneous and supported in $ \Gamma_{\xi_j, \varepsilon} $ so that  $ \psi_j = 1 $ on $ \supp \phi_j $. 
	
	We shall also localize when $| \xi | \ge \varrho \ge 1 $ by $ \chi_\varrho(\xi) = \chi(| \xi|/\varrho ) \in C^\infty $, where  $\chi \in C^\infty(\br) $ such that $ 0 \le \chi \le 1 $,  $ \chi(t) = 0 $ when  $t \le 1 $ and  equal to $ 1 $ when $t \ge 2  $.  Let $\varphi_{j, \varrho} = \chi_\varrho \varphi_j $  and $\psi_{j, \varrho} = \chi_\varrho \psi_j $
	then $ \varphi_{j} - \varphi_{j, \varrho} $ and $ \psi_j - \psi_{j, \varrho} $ are in $ S^{-\infty} $, $ \forall\, j $ and $ \forall\, \varrho \ge 1 $.
	We also have that $ \varrho  \chi_{\varrho} \in S^{1}$ uniformly in $ \varrho \ge 1 $.
\end{rem}

In fact, since  $0 \le  \chi_{\varrho} \le 1$ and $ \varrho \le | \xi | $ in the support of this symbol, we find that  $| \varrho  \chi_{\varrho}( \xi)| \le  |\xi|$. Taking $ \xi $ derivatives  of  $ \varrho  \chi_{\varrho}  $ gives a factor $ \varrho^{-1} $ together with a symbol supported where $\varrho \le| \xi | \le 2\varrho $.

Next, we have to prepare the linear operator $ P $ microlocally with respect to this partition of unity.
We shall use microlocal pseudodifferential operators which may give complex solutions. But in the case when  $ P $ is a real PDO, we may take the real part of the solutions to the linear equations to obtain  real solutions to~\eqref{pequation}. 
In the following we shall denote $ \w{D} =  (1 + |D|^2) ^{1/2} \in \Psi^1$, and
$ I^k $ the classical Fourier integral operators  of order $ k $ with homogenous phase functions and classical symbol expansions.

\begin{prop}\label{normalform}
	Let $ P $ be given by Theorem~\ref{appthm} with 
	$ v  \in C^\infty$  
	and let  $ \Gamma=  \Gamma_{\xi_0, \varepsilon}$ be defined by~\eqref{gammaform} for  $ | \xi_0 | = 1 $,  $0 <  \varepsilon \le \varepsilon_0 $. Then for any $ x = x_0 $ and $ v = v_0$  there exists $ \varepsilon_0 > 0$ small enough and real valued $  0  \ne a(v,x, \xi) \in S^0 $, $ 0 < c \w{\xi}^{1-m} \le b(\xi)  \in S^{1-m} $ 
	 and linear symplectic changes of  variables so that 
	\begin{equation}\label{prepform}
		 P(v,t,x,  D)b(D ) =  	a(v,t,x,  D)  Q(v,t,x, D) + R(v,t,x, D)
	\end{equation}
	where
	\begin{equation}\label{normform}
    Q(v,t,x, D)  = D_{t} + \sum_{j = 1}^{n-1} A_j(v,t,x, D_x) D_{x_j} + A_0(v,t,x,  D_x)    \qquad  (t,x) \in \br \times \br^{n-1}
 	\end{equation}
	Here $ a $, $A_j $ and  $ R $ are operators that depend $ C^\infty $ on $ v(x)$, $A_j  \in  C^\infty(\br,  S^0) $ is real valued when $ j > 0 $ and $ R  = R_0 + R_1 \in \Psi^1 $ where $ R_0 \in \Psi^{-1}$ and $ \wf R_1 \bigcap \Gamma_0 = \emptyset $,  $ \Gamma_0 = (v_0, x_0) \times \Gamma_{\xi_0, \varepsilon} $. The seminorms of $ A_{j} $, $ R $ and the constant~$ \varepsilon_0 $  only depend on the seminorms of  $ v $ and the coefficients of  $ P $.
	If $ P(v,t,x,D) $ is defined for real valued $ v \in C^\infty $ then $ Q $, $ R $ and $ A_j $ will also be defined for  real valued $v \in  C^\infty $.
\end{prop}

\begin{proof}
	For the proof, it is important that compositions of  operators that depend $ C^\infty $ on $ v $ also depend $ C^\infty $ on $ v $, see Lemma~\ref{contlem}.
	Let $ \Sigma = p_m^{-1}(0) \subset T^*\br^n$ and 
	let $ \xi_0\in \br^n  $ so that $|\xi _0 | = 1  $. If $ (v_0, x_0) \times\Gamma_{\xi_0, 2\varepsilon} \bigcap \Sigma  \ne \emptyset$ then, after a linear symplectic changes of  variables,  we find that $ |\partial_\xi p_m| \cong |\xi |^{m-1} $ in  $\Gamma_0 =  (v_0, x_0) \times \Gamma_{\xi_0, \varepsilon} $ for small enough $ \varepsilon $ since $ p_m  $ is of principal type. Then we have $ p_m = \sum_j A_j \xi_j = A\cdot \xi$ with $ |A | \cong | \xi |^{m-1}$ in $\Gamma_0 $. If   $ (v_0, x_0) \times\Gamma_{\xi_0, 2\varepsilon} \bigcap \Sigma  = \emptyset$ then we have $ |p_m| \cong |\xi |^{m} $ in $ \Gamma_{0} $ by homogeneity and compactness. Then we find by homogeneity that $ \xi\cdot \partial_\xi p_m = mp_m  $ in $ \Gamma_{0} $ thus  $ p_m = \sum_j A_j \xi_j$ with $ |A | \cong | \xi |^{m-1}$ in $ \Gamma_{0} $ in this case too. Thus $ \exists \, j $ so that $ |A_j | \cong | \xi |^{m-1}$ near the ray through $(v_0,x_0,\xi_0) $.  After an ON transformation  we obtain that $ |A_1 | \cong | \xi |^{m-1}$ in $ \Gamma_{0} $  by shrinking~$\varepsilon$. 
	
	Letting  $ b (\xi)  =  A_1(v_0,x_0, \xi)^{-1}$  near  $ \Gamma_{0} \bigcap \set{|\xi | \ge 1}$ we may extend $ b(\xi)  $ to a symbol in $ S^{1-m}$ so that $ b(\xi) \gs \w{\xi}^{1-m} $ modulo terms having wave front set outside~$ \ol \Gamma_0 $. Then $ P b(D) $  has symbol expansion with $ p_jb \in S^{j + 1 - m} $ and principal symbol $ p_m(v,x,  \xi) b({\xi})  =  \sum_j A_j(v,x,  \xi)b(\xi)  \xi_j \in S^1$, where $ A_j(v,x,  \xi)b(\xi)  \in S^0$.  Replacing $ A_jb $ with $ A_j  $ we find
	\[ 
	 P b(D) = \sum_{j = 1}^{n} A_j(v,x, D) D_{j} + A_0(v,x, D)
	 \]
	where $ A_0  \in \Psi^0$, $ A_1(v,x,  \xi)  \ne 0$  near $ \Gamma_0 $ and  $ A_j \in S^0$. 
	No we can extend $a = A_1  $  outside $ \Gamma_0 $ so that $0 \ne  a(v,x,  \xi) \in S^0$. 
	Replacing $ A_j $ with $ a^{-1}A_j $  we find that $ Pb(D) = a(v,x, D) Q  \in \Psi^{1}$ where the symbol of $ Q $ is equal to $ \xi_1 + \sum_{ j > 1}A_j(v,x,\xi)\xi_j  $ modulo $ S^{0} $ and terms having wave front set outside $ \ol \Gamma_0 $.
	
	To obtain that $ A_ j$ is independent of $ \xi_1$ for $ j > 0 $, we shall use the Malgrange preparation theorem. If  $ \xi = (\xi_1, \xi') $ we find by homogeneity for small enough $ \varepsilon_0 > 0 $ that
	\begin{equation}
	\xi_1 + \sum_{ j > 1}A_j(v,x,\xi)\xi_j = q(v,x,\xi)\left(\xi_1 + r(v,x,\xi') \right)	
	\end{equation} 
	in a conical neighborhood of $ \Gamma_0 $, where $0 < q \in S^0$ is  homogeneous of degree 0 and $ r  \in S^1$ is real and homogeneous of degree 1. Then we can extend $ q  $ to a positive homogeneous symbol by a cut-off.
	This replaces $ a$ by $ 0 \ne aq \in S^0$, and by using Taylor's formula we find that $r(v,x,\xi') = \sum_{j >1}r_j(v,x,\xi') \xi_j $ with $ r_j $ homogeneous of degree 0 in $ \xi' $. 
	This gives that $ Pb(D) =  aq(v,x,D) Q$   where $ Q $ is equal to~\eqref{normform} modulo $ \Psi^0 $  and terms having wave front set outside $ \ol \Gamma_0 $.
	Observe that the composition $ aq(v,x,D) $ with $ Q $ also gives lower order terms in $ \Psi^0 $ which can be included in $ A_0 $. 
	
	Now the term $ A_0 \in  \Psi^0$ of  $ Q $ can be replaced by $  aqR_0 $ modulo $ \Psi^{-1} $ where the symbol $ R_0 = A_0/aq \in S^0$. To  make the term $ R_0 $ independent of $ \xi_1$ we may use Malgrange division theorem and homogeneity for small enough $ \varepsilon_0 > 0$ to obtain that
	\begin{equation}
	R_0(v, x,\xi) = q_0(v,x,\xi)\left(\xi_1 + r(v, x,\xi') \right)	 + r_0(v, x,\xi')
	\end{equation}
	in a conical neighborhood of $ \Gamma_0 $, where $ q_0  \in S^{-1}$ and $ r_0 \in S^0 $ by homogeneity. Cutting of $ q_0 $ we may replace $ R_0(v,x,D) $ with $ q_0(v,x,D) Q + r_0(v,x,D_{x'}) $ modulo~$\Psi^{-1}$ near $ \Gamma_0 $. Since $ R_0 \cong (1 + q_0)R_0 $ modulo $\Psi^{-1}$, we obtain~\eqref{prepform}--\eqref{normform} with $ a $ replaced by $aq( 1 + q_0)$.
	Cutting off $ q_0 $ where $ |\xi | \gg 1 $ only changes the operator with terms in $ \Psi^{-\infty} $, but gives that $  1 + q_0  > 0$ making $ aq(1 + q_0 )\ne 0 $.  The composition of $ aq \in  \Psi^{0}$ with $ q_0 Q \in  \Psi^{0}$ will also give lower order terms in $ \Psi^{-1} $ which can be included in $ R $ together with any cut-off terms. Observe that the reductions in the proof hold for both real and complex valued parameters $ v $, but lower order terms may be complex. This gives the proposition  after putting $  t= x_1 $ and $ x = x' $.
\end{proof}

\begin{rem}\label{normrem} 
	Proposition~\ref{normalform} shows that there exists $ F_0 $ and $ F_1 \in I^0 $ so that $ F_0F_1 \cong F_1F_0 \cong \id $ modulo $ \Psi^{-2} $ and $ F_0PbF_1 $ is on the normal form~\eqref{prepform}. By multiplying with $ F_1 $ we obtain~\eqref{prepform} with $ a$ replaced by $F_1 a \in I^0 $, $ b $ replaced by $ bF_1 \in I^{1-m} $ which are elliptic, and $  R  $ replaced by $F_1 R \in I^1 $ modulo $ I^{-1} $, which is in $ I^{-1} $  at $ \Gamma_{0} $. 
\end{rem}

We say that $ E \in I^k $ in the open set $ \Omega \subset T^*\br^n $ if  $ E = E_0 + E_1 $ where $ E_1 \in I^k $ and $ \wf E_0 \bigcap \Omega= \emptyset $.
Proposition~\ref{normalform} shows that the linerarized equation $P (v,x,D)u = f  $ may after linear sympectic changes  of variables be microlocally reduced to the system $Q_j (v,x, D)u_j   \cong a_j^{-1}f_j $
where $f_{j} = \varphi_{j} f$ and $ u \cong \sum_j b_ju_j $ with $  b  \in I^{1-m} $ and $  \varphi_{j}$ given by Remark~\ref{microrem}.  Observe that  $ u_j $ has to be microlocalized.

\section{Solutions of the Microlocalized Equation}

Next, we shall solve the microlocalized equations $ Q_ju_j = a_j^{-1} f_{j} = a_j^{-1}  \varphi_{j}f  $. Here  $  \varphi_{j}$ is given by Remark~\ref{microrem}, $ Q_j $ is given by~\eqref{normform} with  $ a_j \in I^{0} $ given by Proposition~\ref{normalform}, but we shall treat the term $R $ as a perturbation.  
In the case when $ A_j \equiv 0 $, we would then find that  $Q_j u_j  \cong D_{t} u_j   = a_j^{-1} f_{j} $, which has the approximate solution $ u_j  \cong  \int_0^{t}a_j^{-1} f_{j} \, dt $. By using Fourier integral operators one can reduce to this case. 

As before, we shall denote by $ I^k $ the classical Fourier integral operators  of order $ k $ with homogenous phase functions and classical symbol expansions, but now depending $ C^\infty $ on~$ v $.
But we shall also use operators $F(t) \in C^\infty(\br, I^{k}) $ so that $F(t) \in I^k $ is FIO in  $  x$ depending $ C^\infty $ on~$ t $ and~$ v $ for $ (t,x) \in \br \times \br^{n-1}$.  Observe that $ \Psi$DO of order $ k$ in $ x $ depending $ C^\infty $ on~$ t $ and~$ v $ are  also in $C^\infty(\br, I^{k})$ and that $C^\infty(\br, I^{k}) \subset I^k$. By multiplying  $ I^k $ by $ I^m $ we obtain operators in $ I^{k + m} $ by Remark~\ref{fiorem}. 

As before, we shall use ON coordinates  $ (t,x) \in \br \times \br^{n-1}$ and suppress the dependence on  $ v $. But the operators will depend $ C^\infty $  on  $ v $ and $ w $ having symbols and phase functions that are uniformly bounded  if $ v \in C^\infty$  and $ w \in \br^m$ are bounded. 

.

\begin{prop}\label{normeq}
Assume that $ Q = D_{t} + a_1(t,x,D_{x}) + a_0(t,x,D_{x})  $ depends $ C^\infty $ on $ v \in C^\infty$,  where $ a_1 \in  C^\infty(\br, S^1) $ is real and homogeneous of degree 1 and $ a_0 \in  C^\infty(\br,  S^0) $. Then there exists elliptic Fourier integral operators $ F_0(t) $ and  $ F_1(t) \in C^\infty(\br, I^{0}) $ such that $ F_0(t)  F_1(t) \cong \id$ and $ Q F_0(t)  \cong F_0(t)  D_t $  modulo $C^\infty(\br, I^{-1})$.
If $ f \in C^\infty_0 $ then we have that
\begin{equation}\label{soleq}
u(t,x) =i F_0(t) \int_0^t F_1(s) f(s,x)\, ds  = \mathbb{ F}f(t,x) 
\end{equation}
solves the initial value problem
\begin{equation}\label{normequation}
Q u  = (\id + S)f  \in C^{\infty} \qquad u(0,x)\equiv 0 
\end{equation}
where $ S \in C^\infty(\br, I^{-1})$ and $ \mathbb{ F} \in I^0 $. 
Here $ F_0(t) $, $ F_1(t) $ and $ \mathbb{ F} $ depend $ C^\infty $ on $ v\in C^\infty $ and have wave front sets close to the diagonal for $ |t| \ll 1 $. In fact, the canonical transformations given by $ F_0(t) $ and  $ F_1(t) $ maps bicharacteristics of $D_t + a_1 $ to $t $ lines and vice versa. 
\end{prop}

\begin{cor}\label{locrem}
If $ c_j \in \Psi^{k} $, $ j = 1,\ 2 $, and $ \supp c_1 \bigcap  \supp c_2  =  \emptyset$, then  \/ $ c_1F_0(t)F_{1}(s)c_2 \in I^{-\infty}$
having a smooth kernel  for small enough $| s| $ and $| t| $.
\end{cor}

\begin{proof}
It is a classical result that there exists elliptic FIO $ F_0(t) $ and  $ F_1(t) \in I^0$  with the properties in the proposition. 
The construction of the homogeneous phase function of the FIO involves solving the Hamilton--Jacobi equations, which depend on the derivatives of the principal symbol $ \tau + a_1 $ of $ Q $ and thus depend $ C^\infty $ on $ v $. Then the amplitude is given by the transport equations which also depend on the lower order term $ a_0 $ of  $ Q $ modulo terms in $C^\infty(\br, S^{-1}) $ and thus depend $ C^\infty $ on $ v $. For the approximate inverse, one takes the phase function for the inverse canonical relation and the inverse amplitude, which also depend $ C^\infty $ on $ v $. 

If $ f \in C^\infty_0 $ and $ u = i F_0(t) v$ with $ v = \int_0^t F_1(s) f(s,x)\, ds $  as in~\eqref{soleq}, then we find 
\begin{multline}
Qu = iQ F_0(t)v = (iF_{0}(t)D_t + S_0 )v =  F_0(t)F_{1}(t) f  +  S_0 v  \\
= f  + S_1f  + S_0 v = f  + S_1f +  S_0\int_0^t F_1(s) f \, ds = f + S f
\end{multline}
	where $ S_0 $, $ S_1 $ and $ S \in C^\infty(\br, I^{-1})$.
\end{proof}

Now  by integrating in $\xi $ we find that $ S\in I^{-\infty} $ has smooth kernel $ S(x,y) $. Then the term $ Sf (x) = \int S(x,y) f(y ) \, dy$ can be made small if $ f $ has support in a sufficiently small neighborhood of  $ x_0 $.   In fact, let $ \phi_\delta(x) = \phi((x-x_0)/\delta) $ where $ 0 < \delta \le 1 $ and $\phi \in C^\infty_0(\br^n) $ such that $ 0 \le \phi \le 1 $, $ \phi $ has support where $ | x| < 2 $ and is equal to 1 when $ | x| \le 1$, so that $\phi(x/\delta) \in C_0^\infty(B_{x_0, 2\delta}) $  if  $ B_{x_0, \delta} = \set{x:  |x - x_0 | \le \delta} $ and $\phi(x/\delta)  = 1  $ in $ B_{x_0, \delta} $.

\begin{lem}\label{estlem}
	Let $ S(x,y) \in C^\infty $ and 
	$S_\delta(x,y) = \phi_\delta(x) S(x,y) \phi_\delta(y)  \in C^\infty_0(B_{x_0, 2\delta}\times  B_{x_0, 2\delta})$.
	The mapping $S_\delta:   C^\infty \mapsto  C_0^\infty(B_{x_0, 2\delta}) $ is given by $ S_\delta f(x) = \int S_\delta(x,y) f(y ) \, dy$, and for $ f \in  C^\infty_0(B_{x_0, \delta}) $ we have $S_\delta f(x) = Sf(x)$ when $ |x| \le \delta $.
	For $ \delta  $ small enough,  $ \id + S_\delta $  has the inverse  $ (\id + S_\delta)^{-1}  = \sum_{j= 0}^\infty (-S_{\delta})^j \cong \id $
	modulo operators with kernels in $C^\infty_0(B_{x_0, 2\delta}\times  B_{x_0, 2\delta}) $. 
\end{lem}

\begin{proof}
	We may assume that $x_0 = 0  $, clearly $S_\delta f(x) = Sf(x)$ if $ \phi_\delta f = f$ and $ \phi_\delta(x) = 1 $.
	If $ f \in C^\infty $  then $ L^\infty $ norm is  $ \mn{S_\delta f}_\infty \le c_n 2^n\delta^{n }\mn{S}_\infty \mn{f}_\infty $.
	By induction we get 
	\begin{equation}
	\mn{S_{\delta}^{j}f}_\infty \le c_n2^n\delta^{n }\mn{S_\delta}_\infty \mn{S^{j-1}_{\delta}f}_\infty \\
	\le c_n^{j}2^{jn}\delta^{jn }\mn{S}^j_\infty \mn{f}_\infty \qquad j > 1
	\end{equation}
	where the kernels of $ S_{\delta}^{j} $ are in $ C^\infty_0(B_{0, 2\delta}\times  B_{0, 2\delta}) $. Thus the series $ \sum_{j= 0}^\infty (-S_{\delta})^j  $ converges in $ L^\infty $ if $ c_n 2^n\delta^{n }\mn{S}_\infty < 1 $. 
	Derivation of  the terms in the series will only give factors $ O(\delta^{-1}) $ so the convergence is in $C^\infty_0(B_{0, 2\delta}\times  B_{0, 2\delta}) $. 
	Then the  inverse $(\id + S_\delta)^{-1} =  \id + \sum_{j= 1}^\infty (-S_\delta)^j \cong \id  $ modulo operators with  kernels  in $C^\infty_0(B_{x_0, 2\delta}\times  B_{x_0, 2\delta}) $.
\end{proof}

\section{A Calculus Lemma}

The approximate solution  $ u $ in~\eqref{soleq} depends $ C^\infty $ on the data $ f $ and  $ v $, but we shall need stronger estimates. For that we shall use the $ L^2 $ Sobolev norms:
\begin{equation}\label{normdef2}
\mn{\varphi}^2_{(k)} = \mn{\w{D}^k\varphi}^2 \qquad \varphi \in C^\infty_0  \qquad \forall \, k\in \br 
\end{equation}
Then the continuity of  $ a(v,x,D) \in \Psi^m $ depend uniformly on $ v\in C^\infty $ in theses spaces.

\begin{lem}\label{contlem}
	If  $ a(v,x, D) \in \Psi^{m_1}$ and $ b(v,x, D) \in \Psi^{m_2}$  depend  $   C^\infty  $ on $ v(x) $ then we find that  $a(v,x,D) b(v,x,D) \in \Psi^{m_1 + m_2}$ also depends  $   C^\infty  $ on $ v(x) $. 
	There exists $ \ell \in \bn $ so that for  any $ a(v,x, D)  \in \Psi^{0}$ depending  $   C^\infty  $ on $ v(x) \in C^\infty $, and for any $ k \in \br$  there exists $ C_{k}(t) \in C^\infty(\br_+) $ so that
	\begin{equation}\label{est1}
	\mn{a(v,x,D)\varphi}_{(k)} \le C_{k}(\mn{v}_{(\ell)})\mn{\varphi}_{(k)} \qquad \forall \, \varphi \in C^\infty_0
	\end{equation}
	There exists $ \ell \in \br $ so that for any  $ a(v,x, D) \in \Psi^{m_1}$ and $ b(v,x, D) \in \Psi^{m_2}$  depending  $   C^\infty  $ on $ v(x) $, and for any $ k  \in \br$ there exists $ C_{k}(t) \in C^\infty(\br_+) $ so that
	\begin{equation}\label{est2}
	\mn{\, [a(v,x,D), b(v,x,D)]\varphi}_{(k)} \le C_{k}(\mn{v}_{(\ell)})\mn{\varphi}_{(k +m_1+m_2-1)}  \qquad \forall \, \varphi \in C^\infty_0
	\end{equation}
	There exists $ \ell \in \br $ so that for  any $ a (v,x,D)  \in \Psi^{m}$ depending  $   C^\infty  $ on $ v(x) $ having real valued symbol modulo $ S^{m-1} $, and for any $ k  \in \br$ there exists $ C_{k}(t) \in C^\infty(\br_+) $ so that
	\begin{equation}\label{est3}
	\mn{\im a(v,x,D)\varphi}_{(k)} \le C_{k}(\mn{v}_{(\ell)})\mn{\varphi}_{(k+m-1)} \qquad \forall \, \varphi \in C^\infty_0
	\end{equation}
	where $ 2i \im  a(v,x,D) =  a(v,x,D) -  a^*(v,x,D)  $ depends $ C^\infty $ on $ v $. Here the functions $ C_k(t) $ only depend on the seminorms of the symbols.
\end{lem}

\begin{proof}
	First we note that by definition any seminorm of order $ k \in \bn$ of $ a(v,x, \xi) \in  S^m $  is bounded  by $ \mn{v}_{C^k} $ for some $k \in \br $. By the Sobolev embedding theorem, the $ C^k $ norm of $ v $ can be bounded by the norm $ \mn{v}_{(k + s)} $ with $ s > n/2 $.
	
	If $ a(v,x,\xi)  \in S^{m_1}$ and $ b(v,x,\xi) \in S^{m_2}$ then $ a(v,x,D) b(v,x,D) = c(v,x,D)$ is given by
	\begin{equation}\label{compform}
	c(v,x, \xi) = e^{i\w{D_{ \xi}, D_{y}}}a(v,x,\xi)b(v,y, \eta)\restr{\substack{y = x \\ \eta =  \xi}}
	\end{equation}
	The mapping $ a, b  \mapsto c $ is weakly continuous on the symbol classes $ S^{m}$ so that any semi\-norm of $ c $ only depends on some seminorms of $ a $ and $ b $, see~\cite[Th.\ 18.4.10$ ' $]{ho:yellow}. 
	(Weak continuity means that the restriction to a bounded set is continuous in the $ C^\infty $ topology, see~\cite[Def.\ 18.4.9]{ho:yellow}.) 
	Thus if $a(v,x,D)  $ and $ b(v,x,D) $ depend  $   C^\infty  $ on $ v \in C^\infty $ then $ c(v,x,D) $ also does.
	Observe that the number of seminorms that is needed does not depend on the symbol classes $ S^{m_j} $. 
	
	We may reduce the estimate~\eqref{est1} to the case $ k = 0 $ by replacing $ a(v,x,D) $ with  $A(x,D) = \w{D}^{k}a(v(x),x,D)\w{D}^{-k} \in \Psi^{0} $ and $ \varphi $ with $ \w{D}^{k}\varphi $. Here the symbol expansion of 
	$$ A(x, \xi) \cong a(v(x),x,\xi) + k D_x a(v(x),x,\xi) \xi \w{\xi}^{-2} + \dots  + A_j(x, \xi) + \dots 
	$$  
	where the term $ A_j(x, \xi) \in S^{-j} $ has the factor $ D^j_x a(v(x),x,\xi) $, $ \forall \, j $. The error term for the expansion $ \sum\limits_{j = 0}^{n} A_j(x,D)$ is $ r_n(x,D) \in \Psi^{-n-1} $. This operator has a $ C^0 $ kernel bounded by $ \mn{v}_{C^{n+1}} $ which gives a bound of $r_n(x,D)  $ on $ L^2 $. The $ L^2 $ norm of $ A_j(x,D) $ is bounded by a  fixed seminorm of $ A_j(x,\xi) $, $ \forall\  j$, see~\cite[Th.\ 18.6.3]{ho:yellow}.  This seminorm in turn depends on some seminorm of $ a(v,x,\xi) $ which gives~\eqref{est1}   for some $ C_k(t) \in C^\infty(\br_+)$. 
	
	Any seminorm of the symbol of the commutator $[a(v,x,D), b(v,x,D)]  \in S^{m_1+m_2 -1}$ depends on the same seminorm of the symbols of the compositions $ a(v,x,D)b(v,x,D) $ and $ b(v,x,D)a(v,x,D) $. These seminorms in turn depend on some  seminorms of $  a(v,x,\xi) $ and $  b(v,x,\xi) $. Thus we obtain~\eqref{est2} from~\eqref{est1} for some $ C_k(t) \in C^\infty(\br_+)$.
	
	If $ a   \in S^{m}$ then the adjoint $ a^*(v, x,D)  $ is given by 
	\begin{equation}
	a^*(v, x,\xi)   = e^{i\w{D_\xi, D_x}}\ol a(v,x,\xi)
	\end{equation}
	which is weakly continuous in the symbol class $ S^{m} $ by~\cite[Th.\ 18.1.7]{ho:yellow}. If $ a  $ is real modulo $ S^{m-1} $ then $\im a(v, x,D) \in \Psi^{m-1}  $. Thus any seminorm of the symbol of $ \im a(v, x,D) $ is bounded by some seminorms of $ a(v, x,\xi) $, giving ~\eqref{est3}  for some $ C_k(t) \in C^\infty(\br_+)$. 
\end{proof}

\begin{rem}\label{fiorem}
	The results of Lemma~\ref{contlem} also holds for $a(v,x,D) \in  \Psi^m $ depending $ C^\infty $ on~$ v $ composed by  $F \in  I^k $ depending $ C^\infty $ on $ v $, e.g., the FIO given by Proposition~\ref{normeq}. Operators in $ I^{-\infty} $ have smooth kernels which are  $ C^\infty $ functions of  $ v $.
\end{rem}

In fact, Theorem 9.1 in \cite{ho:weyl} shows that the conjugation of $ \Psi$DO  with FIO gives symbol expansions similar to~\eqref{compform} after change of variables, see for example (9.2)$ '' $ 
in \cite{ho:weyl}. This result is about Weyl operators, but by Theorem 18.5.10 in \cite{ho:yellow} it can be extended to operators having the Kohn-Nirenberg quantization. This gives a calculus with symbol expansions of classical homogeneous FIO with homogeneous phases and symbols, see pages~441--442 in~\cite{ho:weyl}. For example, if $ a \in \Psi^{m}  $ and $ F \in I^k$  then we have $ \mn{aF \varphi}^2 = \w{F^*a^* a F \varphi, \varphi} $ where $ F^*a^* a F = b\in \Psi^{2(m+k)} $, and a similar result holds for $ \mn{Fa  \varphi}^2$.

For $ S \in  I^{-\infty} $ the $ C^\infty $  dependence means that for any $ k \in \br$ we have  $ S \in  I^{-k} $ depending $ C^\infty $   on $ v $. Since the kernel is obtained by taking the Fourier transform in $ \xi $ of the symbol, we find that the kernel of $ S $ is smooth and is a  $ C^\infty $  function of  $ v $.

\section{The Microlocal Estimate}

Next, we are going to prove estimates for the microlocalized operators. 
Then we will use ON coordinates  $ (t,x) \in \br \times \br^{n-1}$ and  for $ k \in \br $ and $  T > 0$ define the local norms
\begin{equation}\label{normdef1}
\mn{\varphi}^2_{k,T} = \int_{|t| \le T} \mn{\varphi}_{k}^2(t) \, dt\qquad \varphi \in C^\infty_0
\end{equation}
and  $ \mn{\varphi}_{k,j,T} = \mn{ \w{D_t}^j \varphi}_{k,T}$, $ \forall \, j \in \br $, 
with $ \mn{\varphi}_{k}^2(t) =   \int |\w{D_x}^k \varphi(t,x)|^2\ dx$.
The following result gives uniform estimates in the Sobolev norms, for simplicity we only take $ k \in \bn $.

\begin{prop}\label{linest}
	Let $ v,\, f \in C^\infty_0 $ and $ u \in C^\infty $ be a solution to
	\begin{multline}\label{lineq}
	Q(v, t,x,D)u  = \partial_t u + \sum_{j=1}^n A_j (v, t,x,D_x) \partial_{x_j} u \\ + A_0(v, t,x,D_x) u  =  f \qquad u(0, x, w) = 0
	\end{multline}
	where $ A_j \in C^\infty(\br , \Psi^0)$  depends $ C^{\infty} $ on $v $, $ \forall\, j $, and $ A_j $ is real valued modulo $ S^{-1} $ for $  j > 0$. Then there exists $ \ell \in \bn $ so that for any $ k \in \bn$ there exists $ C_k(r) \in C^\infty(\br_+) $ so that
	\begin{equation}
	\mn{\phi u}_{(k)}^2 \le C_{k}(\mn{v}_{(\ell)})
	\mn{f}_{(k)}^2    
	\end{equation}
	if  $ \phi \in C^\infty_0 $ has support where $ |t| \le  1$.
	The estimate only depends on the seminorms of the symbol of\/ $Q  $ and $ \phi $.
\end{prop}

Thus, for any $ k \in\bn$ we get uniform local bounds on $ \mn{\phi u}_{(k)} $ when $\mn{ v}_{(\ell)} $ is uniformly bounded. Now $ Q $ is a differential operator in $ t $ but a $ \Psi $DO in $ x $, so in the proof we shall use Lemma~\ref{contlem} in the $ x $ variables.

\begin{proof}
	Let $ A\partial_x = \sum_{j=1}^n A_j  \partial_{x_j} $  and $ \w{u,u}_{k}(t ) =  \mn{u}_{k}^2(t )$ be the sesquilinear form,
	then 
	\begin{multline}
	\partial_t \mn{u}_{k}^2(t ) = 2 \re  \w{	\partial_t u,u}_{k}(t )	=  2 \re  \w{ f, u}_{k}(t ) \\ - 2 \re  \w{A \partial_t u,u}_{k}(t ) - 2 \re  \w{A_0 u,u}_{k}(t )
	\end{multline}
	Conjugating with $ e^{-Ct} $ gives 
	\begin{multline}
	\partial_t ( e^{-Ct} \mn{u}_{k}^2(t )) 	=  e^{-Ct}  \big(2 \re  \w{f, u}_{k}(t ) \\ - 2 \re  \w{A \partial_x u,u}_{k}(t ) - 2 \re  \w{A_0 u,u}_{k}(t )  - C \mn{u}_{k}^2(t )\big)
	\end{multline}
	where 
	\begin{multline}
	2 \re  \w{A \partial_x u,u}_{k}(t ) = 2 \re \w{[\w{ D_x}^k, A]\partial_x \w{ D_x}^{-k} w, w}_{0}(t ) \\ +  \w{[\re A,\partial_x, ]  w, w}_{0}(t ) + 2 \re \w{i  \im A\partial_x  w, w}_{0}(t )
	= \w{R_k w, w}_{0}(t)
	\end{multline}
	where $ [\re A, \partial_x] $ and $ \im A\partial_x \in C^\infty(\br,  \Psi^{0}) $ and $ w = \w{ D_x}^k u $. The calculus gives that the operator $ [\w{ D_x}^k, A]\partial_x \w{ D_x}^{-k}  \in C^\infty(\br,  \Psi^{0})$, so that  $ R_k\in C^\infty(\br,  \Psi^{0}) $ depends $ C^\infty $ on $ v $. 
	
	Since $ \mn{w}_{0} =  \mn{u}_{k} $  we find by 
	using Lemma~\ref{contlem} that
	\begin{equation}
	| \w{R_k w,w}_{0}(t )| \le  C_k(\mn{v}_{\ell}(t))\mn{w}_{0}^2(t ) = C_k(\mn{v}_{\ell}(t))\mn{u}_{k}^2(t )
	\end{equation}
	for some   $ \ell \in\bn $  and $ C_k(r) \in C^\infty(\br) $, and clearly
	\begin{equation}
	| \w{ f,u}_{k}(t )|   \le  \mn{ f}_{k}^2 (t )+ \mn{u}_{k}^2(t )
	\end{equation}
	We also obtain from Lemma~\ref{contlem} that
	\begin{equation}
	|\w{A_0 u,u}_{k}(t )| \le C_k(\mn{v}_{\ell}(t))\mn{u}_{k}^2(t )
	\end{equation}
	where in the following we will take the maximum of different  $ \ell $ and $ C_k(r) $.
	Summing up, there exists $ C_k(r)  $ so that
	\begin{equation}
	\partial_t ( e^{-Ct} \mn{u}_{k}^2(t )) 	\le  e^{-Ct}\left((C_k(\mn{v}_{\ell}(t))  - C)\mn{u}_{k}^2(t )
	 + 	C_k(\mn{v}_{\ell}(t))\mn{ f}_{k}^2(t)\right) 
	\end{equation}
	Now we may replace $ C_k(r)  $ by a nondecreasing function
	and put 
	$$ C_k =  \max_{|t | \le 1} C_k(\mn{v}_{\ell}(t)) \le C_k\left(\max_{|t | \le 1} \mn{v}_{\ell}(t)\right) 
	$$ 
	where $ \mn{v}_{\ell}(t) \le  C_0\mn{v}_{( \ell + 1)}$,  $\forall \, t$, by Sobolev's inequality. 
	Since  $  \mn{u}_{k}(0 ) = 0$ we find by integrating that
	\begin{equation}
	e^{-C_kt} \mn{u}_{k}^2(t ) 	\le e^{C_k}C_k(\mn{v}_{(\ell + 1)})\mn{ f}_{k,1}^2   \qquad  t \in [-1, 1]
	\end{equation}
	Integrating over $ [-1, 1] $ we obtain that
	\begin{equation}
	\mn{u}_{k,1}^2	\le 2 e^{2C_k}	C_k(\mn{v}_{( \ell + 1)})\mn{ f}_{k,1}^2
	\end{equation}
	where $ \mn{ f}_{k,1} \le \mn{ f}_{(k)} $.
	Renaming  $2 e^{2C_k}C_k(\mn{v}_{( \ell + 1)})  $  as  $C_k(\mn{v}_{( \ell + 1)})  $  and changing $ \ell $ 
	we obtain 
	\begin{equation}\label{kest}
	\mn{u}_{k,1}^2	\le C_k(\mn{v}_{( \ell )}) \mn{ f}^2_{(k)}  \qquad \forall\, k \in \bn  
	\end{equation}

	Next, we shall estimate 
	\begin{equation}\label{kjest}
	\mn{u}_{k,j,1}^2	\le C_{k,j}(\mn{v}_{( \ell )}) \mn{ f}^2_{(k+j)}  \qquad \forall\, k \in \bn 
	\end{equation}
	 when $j  \in \bn$. For that it suffices to estimate $ 	\mn{D_t^{j} u}_{k,1} $ 
	which we shall do by induction. The case $ j = 0 $ is given by~\eqref{kest}, and we assume that~\eqref{kjest} holds for $ i \le j $ for some $ j \in \bn$. Since $ Qu = \partial_t u + Au = f $ with $ A\in C^\infty(\br, \Psi^1) $, we have
	\begin{equation}
			\mn{D_t^{j+1} u}_{k,1} \le	\mn{ D_t^{j}f}_{k,1} + 	\mn{D_t^{j}A u}_{k,1} 
	\end{equation}
	where $D_t^{j}A = A D_t^{j} + \sum_{0 \le i < j} B_{i}D_t^{i} $
	 with $  B_{i}(t,x ,D_x) \in C^\infty(\br, \Psi^1)$ being a $ \Psi $DO in $ x $ depending $ C^\infty $ on $ t $ and $ v $. By using Lemma~\ref{contlem} and integrating  over $ [-1, 1] $ we find that 
	 \begin{equation}
	\mn{D_t^{j}A u}_{k,1} \le	\sum_{0 \le i \le j} C_{k,i}(\mn{v}_{( \ell )})	\mn{ u}_{k,i,1} 
	 \end{equation}
     so the induction hypothesis gives~\eqref{kjest} for any $ j \in \bn$.
	
	Finally, we shall show that
	\begin{equation}\label{finalest}
	\mn{ \phi u}^2_{ (k)}	\le C_k( \mn{v}_{(\ell)}) \mn{ f}^2_{(k)} \qquad \forall\, k \in \bn
	\end{equation}
	if $\phi \in  C^\infty $ is supported where $ |t| \le 1 $.
	To estimate $ \mn{\phi u}_{ (k)}^2 $ it suffices to estimate $ \mn{D_x^\alpha D_t^j  \phi u} $ for $ | \alpha | + j \le k $.
	We have that 
	$$
	[D_x^\alpha D_t^j , \phi  ] =  \sum_{|\beta | + i \le k}B_{\beta, i}D_x^\beta D_t^i 
	$$ 
	where $  B_{\beta, i}\in  C^{\infty} $ has support where $ |t | \le 1 $. 
	Thus, \eqref{kjest} gives that 
	\begin{equation}\label{finallyest}
	\mn{D_x^\beta D_t^j\phi u}	\le C\sum_{0 \le  i \le j} \mn{u}_{k-i,i, 1}	\le   \sum_{0 \le  i \le j} C_{k-i,i}(\mn{v}_{(\ell)}) \mn{ f}^2_{(k)}
	\end{equation}
	which completes the proof. 
\end{proof}

\section{Solutions of the Linearized Equation}

Next, we shall solve the  IVP for the linearized equation
\begin{equation}\label{plinequation}
P(v(x),x, D) u(x) = f(x) 
\end{equation}
where $ f $, $ v \in C^\infty$ 
and $ P $ is on the form~\eqref{psymbol} with principal symbol $ p_m(v(x),x, \xi) $ homogeneous of degree $ m \in \br$ and of real principal type when $ | x | \ll 1 $. 
In the following, we shall suppress the parameter $ v \in C^\infty$, the preparation will only depend on the bounds on this parameter.

To solve equation~\eqref{plinequation}, we shall assume that $ x_0 = 0 $ and use the microlocal normal forms given by Proposition~\ref{normalform}. In fact, for any small enough $ \varepsilon > 0 $  we can by Remark~\ref{microrem} find  a partition of unity $\{ \varphi_{j}(\xi) \}_j $ with $ \varphi_{j}\in S^0 $ supported in cones $ \Gamma_{\xi_j, \varepsilon} $  and linear symplectic variables  so that $  P b_j = a_j Q_j + R_j$ satisfies the conditions in Proposition~\ref{normalform} and Remark~\ref{normrem}, with $\Gamma_0 = \Gamma_j = (v_0,0) \times \Gamma_{\xi_j, \varepsilon} $ after the linear symplectic change of variables. 
Here  $ a_j \in I^0$,  $b_j \in I^{1-m}$  are elliptic and 
$Q_j = D_{x_1} + A_jD_{x'} + A_{0,j} $
satisfies the conditions in Proposition~\ref{normeq}. The remainder term $ R_j \in I^1 $ is also in $I^{-1}$ at  $\Gamma_j $.

Since $I^0 \ni a_j $ is elliptic, there exists $  a_j^{-1} \in I^0$ svch that $ a_j a_j^{-1}  = \id  + B_j$ with $ B_j \in \Psi^{-1} $. Thus, 
by ignoring the operator $ R_j $, which will be handled as a perturbation,  we obtain from Proposi\-tion~\ref{normeq}
that if  $ f \in C^\infty $ then $ u_j = \mathbb{ F}_ja_j^{-1}\varphi_{j}f $ solves
\begin{equation}\label{microeq0}
Q_j u_j = (\id + s_j) a_j^{-1}\varphi_{j}f = (a_j^{-1} + r_j)\varphi_{j}f  \qquad u_j \restr{t = 0} = 0
\end{equation}
where $ s_j $ and $ r_j  = s_ja_j^{-1}\in I^{-1} $.

But $ u_j$ may not be localized near $\Gamma_j  $.
To handle the localization and the error term $ R_j $ we
shall microlocalize $ u_j  $ depending on parameters. We shall use  the cut-off $ \psi_{j,\varrho}(\xi)  = \psi_{j}(\xi)\chi_{\varrho}(\xi) \in S^0$ given  by Remark~\ref{microrem} with $ \varrho \ge 1 $ such that $0 \le  \chi_{\varrho} \le 1$ has support where $ |\xi | \ge \varrho $,  $ \psi_j \varphi_{j} = \varphi_j $ and $ \supp \psi_j \in \Gamma_{j} $. 
We shall also cut off with  $\Phi(x) \in C^\infty_0(\br^n) $ such that $ 0 \le \Phi \le 1 $, $  \Phi $  has support where $ | x| \le 1 $ and is equal to 1 when $ | x| \le 1/2 $.
We find that
\begin{equation}\label{microloc}
u_j =  \psi_{j,\varrho}\Phi u_j  + (1- \psi_{j,\varrho}) \Phi u_j  + (1 - \Phi)u_j=  u_{j,\varrho}  + S_{j,\varrho} f
\end{equation}
when $ | x| < 1/2 $, where $u_{j,\varrho} =    \psi_{j,\varrho}\Phi u_j$ and $ S_{j,\varrho} = (1- \psi_{j,\varrho})\Phi  \mathbb{F}_ja_j^{-1} \varphi_{j}   \in I^0$ since $ \mathbb{F}_j a_j^{-1} \varphi_{j} \in I^0$. Thus we find that $ S_{j,\varrho}f(x) $ depends on the values of $ f(y) $ when $y_1 $ is in the interval between 0 and $ x_1 $.
Since  $ \psi_j \varphi_{j} = \varphi_j $ we find that $  (1- \psi_{j,\varrho}) \varphi_{j} = (1- \chi_{\varrho}) \varphi_{j}  $ which is supported where $ | \xi | \ls \varrho $.

Since $ [Q_j, \Phi] $ is supported where $ |x| \ge 1/2 $ modulo $ I^{-\infty} $,  we obtain from~\eqref{normequation} that
$$ 
Q_jS_{j,\varrho} =   [\psi_{j,\varrho}, Q_j ]  \Phi \mathbb{F}_ja_j^{-1}\varphi_{j} +  (1- \psi_{j,\varrho})\Phi (\id + s_j)a_j^{-1} \varphi_{j}  \in I^{-\infty}
$$ 
when $ | x| < 1/2 $, where $ s_j \in I^{-1} $ so $Q_jS_{j,\varrho} \in I^0 $ uniformly with symbol supported where $| \xi | \ls \varrho $ modulo  $ S^{-\infty}$. 
We find from~\eqref{prepform}, \eqref{microeq0} and~\eqref{microloc} that
\begin{equation}\label{linsol0}
Pb_j u_{j,\varrho} = a_jQ_j (u_{j}  - S_{j,\varrho} f)  + R_j u_{j,\varrho} 
= \left((\id + B_j + a_jr_j)\varphi_j  - a_j Q_jS_{j,\varrho} +  R_{j,\varrho} \right) f  
\end{equation}
where $ a_j Q_jS_{j,\varrho} \in I^{-\infty}$, $ a_jr_j \in I^{-1} $ and $ R_{j,\varrho} =  R_j\psi_{j,\varrho} \mathbb{F}_ja_j^{-1} \varphi_{j}\in I^{-1}$ with symbol supported where $ |\xi | \gtrsim \varrho $ modulo $ S^{-\infty}$ uniformly when $|x | \ll 1 $, since  $ R_j \psi_{j,\varrho} \in I^{-1} $ for $| x| \ll 1 $ by Proposition~\ref{normalform}. 
Since $ \varrho \psi_{j,\varrho} \in S^1 $ uniformly by Remark~\ref{microrem}, we find $ \varrho R_{j,\varrho} \in I^0$  uniformly when $| x| \ll 1$.

Now we define
\begin{equation}\label{linsol}
u_{\varrho}(x) = \sum_{j} b_j(D)\psi_{j,\varrho}(D) \Phi  u_j(x)  = \sum_{j} b_j(D)u_{j,\varrho}(x) 
\end{equation}
where $\varrho^{m-1} b_j\psi_{j,\varrho} \in I^0 $ uniformly when $ \varrho \ge 1 $.
We obtain from~\eqref{linsol0} and~\eqref{linsol} that
\begin{equation}\label{sumsys}
P u_{\varrho} =  
f +  \sum_j \left((B_j  + a_jr_j)\varphi_{j}  - a_jQ_jS_{j,\varrho} + R_{j,\varrho}\right)f
= (\id+ R_{\varrho}) f
\end{equation}
where $ R_{\varrho} =  \sum_j (B_j + a_j r_j)\varphi_j  - a_jQ_jS_{j,\varrho} + R_ {j,\varrho} \in I^{-1}$ when $ |x|  \ll 1 $. We shall localize the first terms in $ \xi $ by writing $ B_j + a_j r_j =  (B_j + a_j r_j)(1 - \chi_{\varrho}) +  {(B_j + a_j r_j)\chi_\varrho} $ which gives 
\begin{equation}\label{RrhoTdef}
R_{\varrho } = 
R_{\varrho,0}  +  R_{\varrho, 1} 
\end{equation}
Here $ R_{\varrho,0} =  \sum_j (B_j + a_jr_j )\varphi_j(1 - \chi_{\varrho})  - a_jQ_jS_{j,\varrho} \in  I^0  $ uniformly with symbol supported where $ |\xi | \ls \varrho $ modulo $ S^{-\infty}$ uniformly when $ \varrho\ge 1 $ and $| x| \ll 1$,  and $  R_{\varrho,1} =  \sum_j (B_j + a_jr_j )\varphi_{j,\varrho} +   R_ {j,\varrho}\in I^{-1}$ uniformly so that  $ \varrho R_{\varrho, 1}  \in I^0$ uniformly when $ \varrho \ge 1 $  and $ |x| \ll 1 $.

Since we are only need local solutions, we may cut off near $ x= 0$.
Let  $ \Phi_\delta(x) = \Phi(x/\delta) $ with $ 0 < \delta \le 1 $. 
To solve the equation near $ x=0 $ it is enough that $ \Phi_\delta Pu_{\varrho} =  \Phi_\delta f $ for some $0 < \delta  \le 1$. 
If  $ f$ has support where $|x| \le \delta/2 $  then $ f = \Phi_\delta f $ so
we obtain from~\eqref{sumsys} that $  \Phi_\delta Pu_{\varrho} =  \Phi_\delta(\id  + R_{\varrho})\Phi_{\delta} f = (\id +  R_{\delta, \varrho}) f$ where $R_{\delta, \varrho} =  \Phi_\delta R_{\varrho}\Phi_{\delta} $. 
By~\eqref{RrhoTdef} we have 
\begin{equation}\label{rdrt}
R_{\delta, \varrho} =   R_{\delta, \varrho, 0} + R_{\delta, \varrho, 1}
\end{equation}
with  $  R_{\delta, \varrho, j} = \Phi_\delta R_{\varrho,j}\Phi_{\delta}$. 
For fixed $0 < \delta \ll 1$ we have  $ \varrho R_{\delta, \varrho, 1} \in I^0$ uniformly when $ \varrho\ge 1 $ and  $R_{\delta, \varrho, 0} \in I^{0}$  has $ C^{\infty} $ kernel depending on  $ \varrho $  and  $ \delta $ with symbol supported where $ |\xi | \ls \varrho $ modulo $  S^{-\infty}$ uniformly when $ \varrho\ge 1 $.

It remains to invert the term $ \id +  R_{\delta, \varrho} $ in order to solve equation~\eqref{plinequation}. 
This will be done in two steps, first making $R_{\delta, \varrho, 1}  $ small by taking large enough $ \varrho $. This may increase the seminorms of $ R_{\delta, \varrho, 0} $, but this term can then be made small by localizing in a sufficiently small neighborhood of  $ x= 0 $.

Since $ \varrho R_{\delta, \varrho, 1} \in I^0$ uniformly when $ \varrho\ge 1 $, we find by Remark~\ref{fiorem} that there exists $\varrho_{\delta}  \ge 1  $  so that 
$ \mn{R_{\delta, \varrho, 1} f}_{(0)} \le \mn{ f}_{(0)}/2 $  for $ f \in \Cal S $  when $ \varrho \ge \varrho_{\delta} $.
Then we find that 
$$ (\id+ R_{\delta, \varrho, 1})^{-1}  = \id + \sum_{ k > 0}(-R_{\delta, \varrho, 1})^k \in I^0 \  \text{ uniformly in $ \varrho $}
$$ 
Observe that $ (-R_{\delta, \varrho, 1})^k $ has kernel supported where $ |x | \le \delta $ and $ |y | \le \delta $.  If we then solve 
\begin{equation}\label{microeq1}
Q_j u_j = a_j^{-1} \varphi_{j}(\id+ R_{\delta, \varrho,1})^{-1}f  \qquad u_j \restr{x_1 = 0} = 0
\end{equation}
for $  f $ supported where $ |x| \le \delta/2 $, then the previous reduction and~\eqref{rdrt} give
\begin{equation}\label{sumsys1}
\Phi_\delta P u_{\varrho} =  (\id+ R_{\delta, \varrho})(\id+ R_{\delta, \varrho,1})^{-1}f = (\id + R_{\delta,\varrho,2}) f 
\end{equation}
where $ R_{\delta,\varrho,2} = R_{\delta,\varrho, 0}(\id + R_{\delta,\varrho,1})^{-1} \in I^{-\infty}$ with $ C^\infty $ kernel supported where $ |x | \le \delta $ and $ |y | \le \delta $ and with symbol supported where $ |\xi | \ls \varrho $ modulo $ S^{-\infty}$ uniformly when $ \varrho\ge 1 $.  Observe  that we have uniform bounds for  fixed $ \delta $ when  $ \varrho \ge \varrho_{\delta} $ and these bounds depend on the bounds on the symbol of $ P $ and the parameters $ v \in C^\infty $ and~$ w $. We shall later put more restraints on the lower bound of $ \varrho $ because of conditions on the estimates, see~\eqref{indest}, and the values of $ u_{\varrho}(0)$ and $ \partial^\alpha u_{\varrho}(0)$ for  $ | \alpha | < m $, see~\eqref{initrho}.

Now we have to shrink the support of $R_{\delta,\varrho,2}  $ to lower the norm of the operator without changing $R_{\delta,\varrho,1}  $. With fixed $0 <  \delta \le 1$ and $ \varrho_{\delta} \ge 1$,  we assume $ \varrho \ge \varrho_{\delta} $ and multiply the equation~\eqref{sumsys1} with $ \Phi_{\delta_0}  $ with $ 0 < \delta_0 \le \delta/2 < 1/2$ so that $  \Phi_{\delta} = 1 $ on $ \supp  \Phi_{\delta_0} $. If  $ f  $ is  supported where $ |x| \le \delta_0/2 < 1/4$, then we obtain as before that  $ f = \Phi_\delta f $ and
\begin{equation}\label{sumsys2}
\Phi_{\delta_0}  P u_{\varrho} =  \left(\id + R_{\delta_0,\varrho,2 }\right)f 
\end{equation}
where $R_{\delta_0,\varrho,2} =  \Phi_{\delta_0} R_{\delta,\varrho,2} \Phi_{\delta_0} $.

\begin{lem}\label{estlem2}
	There exists $C_0 > 0  $ so that if  $ \varrho \ge \varrho_{\delta} $ and $ 0 < \delta_0 \le \delta/2 < 1/2$ we have
	\begin{equation}
	\mn{R_{\delta_0,\varrho,2}\varphi}_{(0) }\le C_0 \varrho^n \delta_0^n  \mn{\varphi}_{(0) } \qquad \forall \,  \varphi \in \Cal S
	\end{equation}
\end{lem}

\begin{proof}
	Now $ (\id + R_{\delta,\varrho,1})^{-1}\in I^{0} $ uniformly in $ \varrho \ge \varrho_{\delta} $ so $ R_{\delta,\varrho,2} = R_{\delta,\varrho, 0}(\id + R_{\delta,\varrho,1})^{-1}\in I^{-\infty}$  having $ C^\infty $ kernel, but we also have that $ R_{\delta,\varrho,2} \in  I^0 $ uniformly in $ \varrho \ge 1 $ with a bounded symbol supported where $ |\xi | \ls \varrho $ modulo $  S^{-\infty}$ uniformly when $ \varrho\ge 1 $.
	
	By integrating the symbol in $ \xi $, we find that $ R_{\delta,\varrho,2} $ has $ C^\infty $ kernel $ R_2 (x,y)$ so that $ \mn{R_2}_\infty \le C\varrho^n $.
	Thus  $ R_{\delta_0,\varrho, 2}$  has kernel  $R_0 = \Phi_{\delta_0}(x) R_{2}(x,y) \Phi_{\delta_0}(y) \in I^{-\infty} $ and since 
	$$ 
	\mn{R_{\delta_0,\varrho, 2}f }^2_{(0)} \le \int \left| \int R_0(x,y) f(y)\, dy \right |^2 \, dx  \le \mn{R_0}^2_{(0)} \mn{f}^2_{(0)} 
	$$
	where  $ \mn{R_0}^2_{(0)} \le C \varrho^{2n}(\int \Phi_{\delta_0}(y)\, dy )^2 \ls \varrho^{2n}\delta_0 ^{2n}(\int \Phi(y)\, dy )^2$, which proves the  lemma.
\end{proof}

By Lemma~\ref{estlem2} we find that if $ c_0 = (2C_0)^{-1/n}$ then for $0 < \delta_0 \le \min(c_0/\varrho, \delta/2) $  we have that $ \mn{R_{\delta_0,\varrho,2}(x,D)\varphi}_{(0) } \le \mn{\varphi}_{(0) }/2$ for $ \varphi \in \Cal S $. We obtain that
$$ 
(\id + R_{\delta_0,\varrho,2}(x,D) )^{-1} = \sum_{j \ge 0}(-R_{\delta_0,\varrho,2}(x,D) )^j \cong \id 
$$ 
modulo  an operator in $ I^{-\infty}$ with $ C^\infty_0 $ kernel supported where $ |x | \le \delta_0 $ and $ |y | \le \delta_0 $. 
By replacing $ f $ in~\eqref{microeq1} by $  (\id + R_{\delta_0,\varrho,2})^{-1} f $ we obtain the first part of the following result.

\begin{prop}\label{initlem}
	Let  $ P(v,x,D) $ be given by Theorem~\ref{appthm}, of real principal type near $ (x_0, 0 )$, $ \varphi_{j} $ be given by Remark~\ref{microrem}, $P b_j  = a_jQ_j  + R_j $ by Proposition~\ref{normalform} with $ x_0 = 0 $ and let $R_{\delta, \varrho,1}  $ and $  R_{\delta_0,\varrho,2}$ be given by~\eqref{rdrt} and~\eqref{sumsys2}  depending $ C^\infty $ on $ v $. Then there exist  $0 <  \delta \le 1 $, $  \varrho_{\delta} \ge 1$ and $ c_0 > 0 $ so that if  $ \varrho \ge \varrho_{\delta} $ and  $ 0 < \delta_0 \le \min (c_0/\varrho, \delta/2) $, then $(\id+ R_{\delta, \varrho, 1})^{-1}(\id + R_{\delta_0,\varrho,2})^{-1} = \id + R_{\varrho,\delta_0}  \in I^0$ is uniformly bounded.
	If	$ f \in C^\infty $ and  $ u_j  \in C^\infty$ solves
	\begin{equation}\label{microeq2}
	Q_j u_j = a_j^{-1} \varphi_j (\id + R_{\varrho,\delta_0})\Phi_{\delta_0}f \qquad u_j \restr{x_1 = 0} = 0 \quad \forall \ j
	\end{equation} 
	then $ u_{\varrho}(x) = \sum_{j} b_j\varphi_{j,\varrho}\Phi u_j(x) $ solves
	\begin{equation}
	P u_{\varrho} = f
	\end{equation}
	when $ |x| \le c_2\delta_0$ for some $ c_2 > 0 $. 
	We also have that $  u_{\varrho}(x) = c_{0,x, \varrho}(f) $ and 
	$$ 
	\partial^\alpha u_{\varrho}(x) = c_{\alpha, x,\varrho}(f) \qquad \text{when  $ | \alpha| < m $ and  $ | x | \le c_2\delta_0 $}
	$$
	where $\varrho^{m - 1}  c_{0,x, \varrho}$ and $\varrho^{m - 1 - |\alpha |} c_{\alpha, x,\varrho}  \in \Cal E' $, $  | \alpha| < m $,  are $ o(1)$ uniformly when $ \varrho\to \infty $, i.e.,  $\varrho^{m - 1 - |\alpha |} c_{\alpha, x, \varrho} (f)  = o(1) $ when $ | x | \le c_2\delta_0 $ and $ \varrho\to \infty $. 
\end{prop}

\begin{proof}
	In only remains to prove the statements about $\partial^\alpha u_{\varrho}(x) $ when $ | x | \le c_2\delta_0 $. Observe that these statements imply that $ P $ is of real principal type near  $ (u_{\varrho}(x), x) $ when $ | x | \le c_2\delta_0 $  and  $ \varrho \gg 1$.
	Since $ u_{\varrho} = \sum_{j} b_ju_{j,\varrho}$ it suffices to consider the terms $ b_ju_{j,\varrho}  = E_{j,\varrho}\Phi u_j$, where  $E_{j,\varrho} = b_j\psi_{j,\varrho} \in I^{1-m} $ has symbol supported where $  |\xi | \ge \varrho  $. 
	By Proposition~\ref{normeq} we have that $ u_j(x) = F_{j,\varrho}f(x) $ depends linearly on $ f $, where $F_{j,\varrho} =  \mathbb{F}_j a_j^{-1} \varphi_{j} (\id+ R_{\varrho,\delta_0})  \Phi_{\delta_0 } \in I^0 $ uniformly when  $ \varrho \ge \varrho_{\delta} $. 
    We only have to prove the continuity, since the factor $ \Phi_{\delta_0 } $ in $ F_{j,\varrho} $ gives compact support.
	   
	 We have that  $ \varrho^{m - 1} E_{j,\varrho}$ and $\varrho^{m - 1 - | \alpha| }  \partial_x^{\alpha} E_{j,\varrho}   \in I^{0}$ uniformly with symbol supported where  $ |\xi | \ge \varrho $. Thus,  when $| \alpha| < m  $ we find
	    \[ 
	    \varrho^{m - 1 - | \alpha| }  \partial_x^{\alpha} E_{j,\varrho}   u(x) = (2\pi)^{-n} \int e^{i(\w{x,\xi} + \phi_j(x,\xi) )} e_{j,\alpha, \varrho}(x,\xi)\wh u(\xi)\, d\xi \qquad \forall\  u \in  \Cal S  \quad \forall\, x    
	     \]
		where $\phi_j(x,\xi)  $ is real and homogeneous of degree $ 1 $ in $ \xi $ and $e_{j,\alpha, \varrho} \in S^0  $ has support where $ | \xi | \ge  \varrho  $. Thus,  we find that
	    \begin{equation}\label{FL1est}
		|\partial_x^{\alpha}E_{j,\varrho}   u(x)| \ls  \varrho^{ | \alpha| + 1 - m } \int_{| \xi | \ge \varrho}  |\wh{ u}(\xi)| \, d\xi \le o(\varrho^{ | \alpha| + 1 - m })\mn{\wh{u}}_{L^1}  \qquad \forall\  u \in \Cal S  \quad \forall\, x    
		\end{equation}
		by dominated convergence as $ \varrho \to \infty $. 
		Here $\mn{\wh u}_{L^1} =  \mn{u}_{FL^1} $ is the Fourier $ L^1 $ norm, so there exists $ C_n > 0 $ so that $ \mn{u}_{FL^1}  \le C_n \mn{u}_{\left(\frac{n + 1}{2}\right)}$ for any $ u \in \Cal S   $. 
	This gives that $	| \partial^{\alpha}E_{j, \varrho}\Phi u_j(x) |  \ls o(\varrho^{| \alpha| + 1 -m}) \mn{\Phi u_j}_{\left(\frac{n + 1}{2}\right) }$ for $ | \alpha| < m  $ and 
	\begin{equation}
	| E_{j, \varrho} \Phi u_j(x) | \le o( \varrho^{1-m }) \mn{  \Phi u_j}_{\left(\frac{n + 1}{2}\right)}  \qquad \text{when $ \varrho \to \infty $ and $  | x | \le c_2\delta_0  $}
	\end{equation}
	Now $ u_j(x) = F_{j,\varrho}f(x) $, where $ F_{j,\varrho} \in I^0  $ uniformly when $ \varrho \ge \varrho_{\delta} $, so 
    by Lemma~\ref{contlem} we find that 
	$$ 
	\mn{\Phi u_j}_{\left(\frac{n + 1}{2}\right)} \le C_0 \mn{ u_j}_{\left(\frac{n + 1}{2}\right)} \le C_1  \mn{f}_{\left(\frac{n + 1}{2}\right)} \qquad \text{when  $ \varrho \ge \varrho_{\delta} $}
	$$  
	which finishes the proof of Proposition~\ref{initlem}.
\end{proof}

\section{Proof of Theorem~\ref{appthm}} \label{thmpf}
	
To solve~\eqref{pequation} we may first assume $ X = \br^n $ and $ x_0 = 0 $, and that $ P(u,x,D) $ is of real principal type near $ (\ol u, 0) $, $ \ol u =    \{  \partial^\alpha u(0) \}_{  |\alpha | < m}$. 
We can then reduce to the case when $ \ol u = 0 $ by changing the dependent variable
\begin{equation}\label{reduc}
u(x) = v(x) + \sum_{|\alpha | < m} u_\alpha x^\alpha/ \alpha!
\end{equation}
in~\eqref{pequation}. Then we obtain the following equation for $ v $:
\begin{equation}\label{pmodequation}
P_{0}v  = f   - P_0\sum_{|\alpha | < m} u_\alpha x^\alpha/ \alpha!
=f - \sum_{|\alpha | < m} u_\alpha P_0x^\alpha/ \alpha! = f_0
\end{equation} 
and $ \partial^\alpha v(0) = 0 $ for $ |\alpha | < m $. Here
\begin{equation}\label{p0eq}
P_0(v,x,D) = P(u,x,D)= P\big(v + \textstyle{ \sum_{|\alpha | < m}} u_\alpha x^\alpha/\alpha !, x, D\big) 
\end{equation}
so $ P_0(v,x,D) $ is of real principal type in a neighborhood $ \omega_0 $ of  $ (0, 0) $.
Now the right hand side of~\eqref{pmodequation} depends on both $ f$ and $u_\alpha $. 
In the case of real operators and real valued data, we get a real change of the dependent variable.

Since we shall only construct local solutions near $ x = 0 $, we shall cut off $v $ with   $ \Phi \in C^\infty_0$ that is supported where $ | x  | \le 1 $,  $0 \le  \Phi \le 1$ and  $ \Phi(x)= 1 $ when $ |x | \le 1/2 $.
	Starting  with  $ v^0 \equiv 0 $  we shall solve the linear equation
	\begin{equation}\label{plinequation0}
	P_0( \Phi v^j(x),x, D) v^{j+1}(x) = f_0(x)  \qquad \forall\, j \ge 0
	\end{equation}
	near $ x = 0 $ so that   $\varrho^{m - 1 - | \alpha | }\partial^\alpha v^{j+1}(0) =  o(1)$, $ | \alpha | < m $,  uniformly as $ \varrho \to \infty $  depending linearly on $ f_0 $.

	To  microlocalize, we use Propositions~\ref{normalform} and~\ref{initlem} to find  $  \varrho_{\delta} \ge 1$, $ 0 < \delta \le 1 $ and $ c_0 > 0 $ so that if $ \varrho \ge \varrho_{\delta}$, $0 < \delta_0 \le  \min (c_0/\varrho, \delta/2) $, we find that~\eqref{plinequation0} after a linear symplectic transformation reduces to the coupled system of equations given by~\eqref{microeq2} for $ j \ge 0 $: 
	\begin{equation}\label{plinsystem}
	Q_k( \Phi  v^{j}(x), x,D) v_k^{j+1}  = a_k^{-1} \varphi_k (\id + R_{\varrho, \delta_0})\Phi_{\delta_0}f_0  \qquad  \text{$ \forall \,k $}
	\end{equation}
	with $ v^0 \equiv 0 $. Here   $ a_k^{-1} $ and $R_{ \varrho,\delta_0}   \in I^0$ uniformly  depending $ C^\infty $ on $ \Phi   v^{j}( x)$ and
	$$ 
	v^{j}(x)=  \sum_{k= 1}^N b_k\psi_{k, \varrho}(D) \Phi v^j_k(x) \in  C^\infty \qquad  j > 0 
	$$
	where  $\varrho^{m-1 - | \alpha |} \partial^\alpha b_k\psi_{k, \varrho} \in I^{0} $ uniformly when  $ | \alpha | < m $ and $ \varrho \ge 1$, since $ b \in I^{1-m} $ and $ \varrho^{m-1- | \alpha |}\partial^\alpha \chi_\varrho \in  \Psi^{m-1}$ uniformly.
	
	Now $ \Phi_{\delta_0 }(x) = 1$  when $ |x| < \delta_0/2 $, so by solving~\eqref{plinsystem} using Proposition~\ref{normeq} we obtain from Proposition~\ref{initlem} a solution to~\eqref{plinequation0} when $ |x| < c_2\delta_0 $ for some $ c_2 > 0 $  such that $ \varrho^{| \alpha | }\partial^\alpha v^{k+1}(x)$, $| \alpha | < m $, are  distributions  of $ f_0  $ which are $  o(1)$ as $ \varrho \to \infty $ when $  |x| \le c_2\delta_0 $. Thus, for large enough $ \varrho $ we find that the graph of  $\{\partial^\alpha v^{k+1} \}_{| \alpha | < m} $ is in $ \omega_0 $ when $  |x| \le c_2\delta_0 $, thus $ P_0(\Phi v^{k+1}(x),x,D) $ is of principal type when $  |x|  \le c_2\delta_0 $. 
	In the case of real operators and real valued data, we shall replace $  v^{j}(x) $ with $ \re v^{j}(x) $.
	
	We are going to prove that the solutions $ v^{j+1} $ to~\eqref{plinequation0} are uniformly bounded in $  C^\infty $ near $ x = 0 $, then we can use the Arzela-Ascoli theorem to get convergence of  a subsequence to a solution to the nonlinear equation~\eqref{pmodequation}.
    Observe that we may solve~\eqref{plinsystem} in any neighborhood of $ x= 0 $, but this system will only give a solution to~\eqref{plinequation0} in a fixed neighborhood of  	$ x= 0 $ depending on $ \delta_0 $ where $ P_0 $ is of principal type.
	
	First we obtain from Lemmas~\ref{contlem} and Remark~\ref{fiorem} that  there exists $ \ell \in \bn $ so that for any $ \mu \in \bn $ there exists $ C_\mu(t) \in C^\infty(\br_+)$ so that 
	\begin{equation}\label{auxeq}
	\mn{a_k^{-1} \varphi_k  (\id+ R_{\varrho,\delta_0})\Phi_{\delta_0}f_0 }^2_{(\mu)} \le C_\mu(\mn{ \Phi  v^j}_{(\ell)})\mn{f_0}^2_{(\mu)}\qquad \forall \, k
	\end{equation}
	since the operators are in $ I^0 $ uniformly depending $ C^\infty $ on $  \Phi  v^j (x)$.

	By~\eqref{plinsystem}, \eqref{auxeq} and Proposition~\ref{linest} we also find that
	there exists $ \ell \in \bn $ such that for any $ \mu \in \bn $ there exists $ C_\mu(t) \in C^\infty(\br_+)$ so that 
	\begin{equation}\label{recureq}
	\mn{\Phi  v_k^{j+1}}_{(\mu)}^2 \le C_\mu(\mn{ \Phi  v^j}_{(\ell)})
	\mn{f_0}_{(\mu)}^2 \qquad \forall \, k
	\end{equation}
	since $ \Phi   \in C^\infty_0 $ has support when $  | x |  \le 1$. 
    Here we may assume that $ C_\mu(t) $ is a nondecreasing function for any $ \mu $.
	Now  $ \varrho^{m-1} b_k \psi_{k, \varrho} \in I^0 $ uniformly when $ \varrho \ge 1$ and  $ \mn{\re u}_{(\mu)} \le  \mn{ u}_{(\mu)} $ for $ u \in \Cal S $, which gives
	\begin{equation}\label{stripest}
	\mn { \Phi v^j}_{(\mu)} \le \varrho^{1- m}  \sqrt{\wt C_\ell}  \sum_{ 1\le k \le N } \mn{  \Phi  v^j_k}_{(\mu)}  \qquad \forall \, \mu \in \bn \quad \forall \, \varrho \ge 1
	\end{equation}
	By using~\eqref{stripest} for $ \mu = \ell  $  we find from~\eqref{recureq} that for any $ k $ we have
	\begin{equation}\label{indest}
	\mn{\Phi  v_k^{j+1}}_{(\mu)}^2\le  C_\mu\left(\varrho^{1- m} \sqrt{\wt C_\ell} \sum_{k=1}^N \mn{\Phi  v^j_k}_{(\ell)} \right) \mn{f_0}_{(\mu)}^2 \qquad \forall \, \mu \in \bn \quad \forall \, \varrho \ge 1
	\end{equation}

	Thus, for any $ \mu \in \bn$ we will obtain uniform bounds on $   \mn{\Phi  v_k^{j}}_{(\mu)}$  if we have uniform bounds when $\mu =  \ell $. Since $ v^{0}_k = 0,\  \forall\, k,$ we find by taking $ \mu = \ell $ in \eqref{recureq} that 
	\begin{equation}\label{est0}
	\mn{\Phi  v_k^{1}}_{(\ell)}^2\le C_\ell(0) \mn{f_0}_{(\ell)}^2 \qquad \forall \, k
	\end{equation}
	where $C_\ell(0) \le C_\ell(1)  $ since $ C_\ell(t) $ is nondecreasing. If we assume for some $ j \ge 1$ that
	\begin{equation}\label{inestj}
	\mn{\Phi  v_k^{j}}_{(\ell)}^2\le C_\ell(1) \mn{f_0}_{(\ell)}^2 \qquad \forall \, k
	\end{equation}
	then by choosing $  \varrho^{m-1} \ge N \sqrt{\wt C_\ell C_\ell(1)} \mn{f_0}_{(\ell)} = \varrho_0^{m-1}$  we find using~\eqref{indest} with $ \mu = \ell $   that \eqref{inestj} holds with $ j $ replaced by $ j +1 $. Since this is true for $ j = 0 $ we obtain by induction that \eqref{inestj} holds  for any $ j $. By~\eqref{indest} we obtain for any $ \mu $  uniform bounds on $  \mn{\Phi  v_k^{j}}_{(\mu)}$  for any $j,\, k $, which by~\eqref{stripest} gives
	that $   \mn{\Phi  v^{j}}_{(\mu)}$ is uniformly bounded for any $ j $. 
	Since $ \mn{\re \varphi}_{(\mu)} \le  \mn{\varphi}_{(\mu)} $ for any $ \mu \in \br $ we also get uniform  bounds on $   \mn{\Phi \re v^{j}}_{(\mu)}$.

	By the Arzela-Ascoli theorem there exists a  subsequence $ \{ v^{k_i}\}_{k_i} $ that converges in $ C^\infty $  to a limit $ v $ on $\Phi ^{-1}(1)  $, i.e., when $ |x| \le 1/2 $, as $ k_i \to \infty $.  When $ v^{k_i} $ is real valued, we get a real limit $ v $.
	By taking the limit of the equation~\eqref{plinequation0} we find by continuity that 
	\begin{equation}\label{limequation}
	P_0(v(x),x, D) v(x) = f_0(x) 
	\end{equation}
	when $ | x| < c_2 \delta_0 $.
	We also obtain   that $ v(x) = c_{0,x,\varrho}(f_0) $ and $\partial^\alpha v(x) = c_{\alpha,x,\varrho}(f_0) $ where  $ \varrho^{m - 1} c_{0,x, \varrho}$ and $ \varrho^{m - 1 - | \alpha |} c_{\alpha,x,\varrho}  \in  \Cal E'$,  $| \alpha | < m  $, are $ o(1) $ uniformly when $ | x | \le c_2 \delta_0 $ and $ \varrho\to \infty $. 
	Thus, for $ \varrho  \ge \varrho_{1} $ large enough,  $ P_0(v(x),x, D) $ is of real principal type when $ | x| < c_2 \delta_0  $.

    Thus, if $ \varrho  \ge \max( \varrho_{0},\varrho_{1},\varrho_{\delta}) $ we have a solution the original equation~\eqref{pequation} when $ | x| < c_2 \delta_0 $ with $ v $ replaced by $u =  v + \sum_{|\alpha | < m} u_\alpha x^\alpha  $ but the condition that $ \partial^\alpha  u(0)= u_\alpha $ is only approximatively satisfied. In fact, by~\eqref{pmodequation} we have 
$$ 
f_0 = \Phi_{\delta_0 }f  - \sum_{|\alpha | < m} u_\alpha P_0 \Phi_{\delta_0 }x^\alpha/\alpha !
$$
and by linearity we find that 
$$  
\partial^\alpha  u(0)  = u_\alpha +  c_{\alpha,0,\varrho}(\Phi_{\delta_0 }f )- \sum_{|\alpha | < m} u_\alpha c_{\alpha,0,\varrho}(P_0 \Phi_{\delta_0 }x^\alpha)/\alpha !
$$ 
but $ P_0 $ depends on $ v + \sum_{| \alpha| < m} u_\alpha x^\alpha /\alpha!$ and thus on $\ol u= \set{u_\alpha}_{| \alpha| < m}$. 
Observe that the neighborhood $ \omega_0 $ is translated to a neighborhood $ \omega_{\ol u} $ of $( \ol u, 0)  $.

If we replace $  \ol u $ with indeterminate $\ol w = \set{w_\alpha}_{| \alpha| < m} $ in~\eqref{reduc} we get the square nonlinear system
\begin{equation}\label{initrho}
\partial^\alpha u(0) = w_\alpha +  d_{\alpha,\varrho} \ol w + e_{\alpha,\varrho} = u_\alpha \qquad | \alpha| < m
\end{equation}
where the coefficients $\set{ d_{\alpha,\varrho} }_{| \alpha| < m}$ and $\set{ e_{\alpha,\varrho} }_{| \alpha| < m}$ are $ o(1) $ uniformly when  $ \varrho \to \infty$ but $ d_{\alpha,\varrho} $ depends $ C^\infty $ on $ \ol  w $. 
Thus the system~\eqref{initrho}  converges to the identity when $ \varrho \to \infty $. Now, by mapping degree arguments, there exists  $ \varrho_{2} \ge 1$ so that  when $ \varrho  \ge \varrho_{2} $ there exists a solution  $ \ol w  $ to~\eqref{initrho} so that $ |  \ol w -  \ol  u | = o(1)$  as $\varrho \to \infty$. 

Thus, there exists $ \varrho_{3} \ge 1$ so that by solving~\eqref{pmodequation} with $ u_\alpha $ replaced by $ w_\alpha $ when $ \varrho  \ge \max( \varrho_{0}, \varrho_{1},\varrho_{2},\varrho_{3},\varrho_{\delta}) $ we obtain  a solution to~\eqref{pequation} when  $ | x| < c_2 \delta_0 $. In fact, if $ \varrho $ is large enough then the graph of $\{\partial^\alpha u \}_{| \alpha | < m} $ for the solution $ u $ will be in $ \omega_{\ol u} $ and $ P(u,x,D) $ will be of real principal type when  $ | x| < c_2 \delta_0 $.

In the case of real operators and real valued data, taking the real part of~\eqref{initrho} with real valued $ w_\alpha $ gives a real square system in $ \ol w $ and a real solution to~\eqref{pequation} for large $ \varrho $. This  finishes the proof of Theorem~\ref{appthm}.
\qed

\bibliographystyle{plain}

\end{document}